\documentclass[12pt,notitlepage,a4paper]{article}
\usepackage{amsfonts}
\usepackage{amsmath}
\usepackage{amssymb}
\usepackage[singlespacing]{setspace}
\usepackage{geometry}

\newtheorem{theorem}{Theorem}

\newtheorem{example}[theorem]{Example}

\newtheorem{remark}{Remark}

\newenvironment{proof}[1][Proof]{\noindent\textbf{#1.} }{\ \rule{0.5em}{0.5em}}
\input{tcilatex}
\begin{document}

\title{{\large Exact Invariant Solutions for Generalized Invicid Burgers'
Equation with Damping}}
\author{{\large Muhammad Alim Abdulwahhab\smallskip } \\
{\large Department of Mathematics and Statistics}\\
{\large King Fahd University of Petroleum and Minerals}\\
{\large Dhahran 31261, Saudi Arabia.}\\
mwahab@kfupm.edu.sa}
\maketitle

\begin{abstract}
In this work \ we study the Lie group analysis of the equation $%
u_{t}+g\left( u\right) u_{x}+\lambda h\left( u\right) =0$ which is one of
the generalization of the classical invicid Burgers equations $\left(
\lambda =0\right) $. \ Seven inequivalent classes of this generalized
equation were classified and many exact and transformed solutions were
obtained for each class.
\end{abstract}

\section{Introduction}

Damping can be regarded as an effect that tends to stabilize or reduce other
effects in a given system such as friction in a mechanical system. \ Damping
devices are used in almost all the systems that we know, thus the importance
of having an equation that models appropriate damping systems cannot be
overstated. \ In this paper, we shall use Lie symmetry to perform analysis
on a generalized invicid Burgers' equation with damping in the form%
\begin{equation}
u_{t}+g\left( u\right) u_{x}+\lambda h\left( u\right) =0  \label{34}
\end{equation}%
where $\lambda \neq 0$ is a constant and $g\left( u\right) $ and $h\left(
u\right) $ are smooth functions of $u=u(x,t)$ in the domain of definition.
Numerous applicable cases of equation $\left( \ref{34}\right) $ are
available in the literature, for instance, researchers have used it to model
the Gunn effect in semi conductors, rotating thin liquid films, chloride
concentration in the kidney and flow of petroleum in underground
reservioirs, see \cite{7} for all the references. \ If $h\left( u\right)
=O\left( u^{\alpha }\right) $, $\alpha >0$, $0<u<<1$, Murray (see \cite{6})
shows that a finite initial disturbance zero outside a finite range in $x$
decays under certain conditions and for each condition, the asymptotic speed
of propagation of the discontinuity was given together with its role in the
decay process. \ Also for $\lambda =1$, $g=u$ and $h=u\left( u-1\right) $,
equation $\left( \ref{34}\right) $ becomes the limiting case of the
Burgers'-Fishers equation%
\begin{equation*}
u_{t}+uu_{x}+u\left( u-1\right) =\frac{\delta }{2}u_{xx}
\end{equation*}%
which has been used to model many physical phenomenon in Mathematical
Biology and Genetics.\newline
Symmetry analysis based on local transformation group is one of the most
powerful and prolific methods used for solving nonlinear partial
differential equations (PDEs). \ Its application to the study of PDEs was
laid down by Sophus Lie, a Norwegian Mathematician, in the later half of the
nineteenth century. \ A symmetry of a system of differential equations is a
transformation that maps any solution to another solution of the system. \
Such transformations are groups that depend on continuous parameters and
consist of either point transformations, acting on the systems' space of
independent and dependent variables, or, more generally, contact
transformations, acting on the space of independent and dependent variables
as well as on all first derivatives of the dependent variables. Lie showed
that the algebra of all vector fields (infinitesimal generators) that leave
a given system of PDEs invariant could be found by solving over-determined
(or under-determined) auxiliary system of linear homogeneous PDEs, known as
the determining equations of the group. \ The calculations involved in
obtaining these determining equations and their possible solutions are
cumbersome, hence researchers have written codes for implementing Lie's
algorithms in various packages for symbolic computations like Mathematica,
Maple, etc. Though these softwares help in circumventing the tedious tasks
encountered in obtaining similarity solutions, they fall short of providing
any tangible results when the underlying PDE has an arbitrary function like
that in $\left( \ref{34}\right) $. \ The key feature of a Lie group, which
makes it very useful, is the parametric representation of smooth functions
on a continuous open interval in $\epsilon $; the group parameter. This
ensures that the mapping is differentiable and invertible and that the
mapping functions can be expanded in a Taylor series about any value of\ $%
\epsilon $. \ For detail ramifications on the subject, see the monographs 
\cite{2}- \cite{5}.\newline
This work is organized as follows. \ In section two, we carry out the group
analysis of equation $\left( \ref{34}\right) $ and obtain seven inequivalent
classes where both $g\left( u\right) $ and $h\left( u\right) $ are
restricted to order of powers of $u$. \ In section three, we present
theorems on transformed solutions for these classes and show how they can be
use to generate non-trivial many-parameters invariant solutions of equation $%
\left( \ref{34}\right) $ from the trivial solution $u\left( x,t\right) =0$.
\ Section four deals with invariants solutions.while conclusion is presented
in section five.

\section{Group Classification}

The generalized invicid Burgers' equation with damping $\left( \ref{34}%
\right) $ is said to be invariant is invariant under the first order
prolonged symmetry operator $X^{\left( 1\right) }=\xi \frac{\partial }{%
\partial x}+\tau \frac{\partial }{\partial t}+\eta \frac{\partial }{\partial
u}+\eta ^{x}\frac{\partial }{\partial u_{x}}+\eta ^{t}\frac{\partial }{%
\partial u_{t}}$ if and only if%
\begin{equation}
\left. X^{\left( 1\right) }\left[ u_{t}+g\left( u\right) u_{x}+\lambda
h\left( u\right) \right] \right\vert _{\left( \ref{34}\right) }=0  \label{35}
\end{equation}%
where%
\begin{eqnarray*}
\eta ^{x} &=&\eta _{x}+\left( \eta _{u}-\xi _{x}\right) u_{x}-\tau
_{x}u_{t}-\xi _{u}u_{x}^{2}-\tau _{u}u_{x}u_{t}\text{,} \\
\eta ^{t} &=&\eta _{t}-\xi _{t}u_{x}+\left( \eta _{u}-\tau _{t}\right)
u_{t}-\xi _{u}u_{x}u_{t}-\tau _{u}u_{t}^{2}\text{.}
\end{eqnarray*}%
Expanding equation $\left( \ref{35}\right) $ leads to under-estimated system
of partial differential equations in $\xi \left( x,t,u\right) $, $\tau
\left( x,t,u\right) $, and $\phi \left( x,t,u\right) $ which are the
coefficients of symmetry generator to be determined. \ Due to this
predicament, group classifications of equation $\left( \ref{35}\right) $ is
very difficult, if not impossible, to achieve. To circumvent this, we shall
assume $\xi _{u}=\tau _{u}=0$ which , when applied in the expansion of
equation $\left( \ref{35}\right) $, results in the following equations%
\begin{gather}
\lambda h^{\prime }\left( u\right) \phi +\lambda h\left( u\right) \tau
_{t}+\lambda g\left( u\right) h\left( u\right) \tau _{x}-\lambda h\left(
u\right) \phi _{u}+\phi _{t}+g\left( u\right) \phi _{x}=0\text{,}  \label{36}
\\
g^{\prime }\left( u\right) \phi -\xi _{t}-g\left( u\right) \xi _{x}+g\left(
u\right) \tau _{t}+g^{2}\left( u\right) \tau _{x}=0\text{.}  \label{37}
\end{gather}%
Differentiating equation $\left( \ref{37}\right) $ twice and re-arranging,
we get%
\begin{equation}
\phi =\frac{1}{g^{\prime }\left( u\right) }\left( g^{2}\left( u\right)
A\left( x,t\right) +g\left( u\right) B\left( x,t\right) +C\left( x,t\right)
\right)  \label{38}
\end{equation}%
subject to the consistency conditions%
\begin{eqnarray}
\tau _{x} &=&-A\left( x,t\right) \text{,}  \label{39} \\
\xi _{x}-\tau _{t} &=&B\left( x,t\right) \text{,}  \label{40} \\
\xi _{t} &=&C\left( x,t\right) \text{,}  \label{41}
\end{eqnarray}%
where $A\left( x,t\right) $, $B\left( x,t\right) $, and $C\left( x,t\right) $
are arbitrary functions of $x$ and $t$. \ Using equations $\left( \ref{38}%
\right) -\left( \ref{41}\right) $ into $\left( \ref{36}\right) $ leave us
with a very complex equation which is absolutely difficult to solve if $g$
and $h$ are left as arbitrary functions of $u\left( x,t\right) $. \ But much
can be achieved if both $g$ and $h$ are considered as functions of order of $%
u$ which lead us to the following results.

\begin{theorem}
\label{damping}Let $g\left( u\right) =u^{k}$, $h\left( u\right) =u^{m}$
where $k$ and $m$ are real numbers such that $k\neq m$ and both not equal to
zero. \ Then we have the following exclusive (inequivalent) cases for $k$.%
\newline
\textbf{Case 1}. \ For $k\neq \pm \left( m-1\right) $, $k\neq \pm \frac{m-1}{%
2}$, $k\neq $ $\frac{m-1}{3}$ and $m\neq 1$, equation $\left( \ref{34}%
\right) $ admits three dimensional symmetry algebra spanned by the closed
vector fields%
\begin{equation*}
M_{1}=\partial _{x},\text{ \ }M_{2}=\partial _{t},\text{ \ }M_{3}=\frac{k-m+1%
}{k}x\partial _{x}+\frac{1-m}{k}t\partial _{t}+\frac{1}{k}u\partial _{u}%
\text{.}
\end{equation*}%
\textbf{Case 2}. \ For $k=m-1\left( m\neq 1\right) $, equation $\left( \ref%
{34}\right) $ admits eight parameters group of projective transformations
spanned by the closed vector fields $M_{1}$,$M_{2}$,%
\begin{eqnarray*}
M_{3} &=&t\partial _{x}+\lambda \left( m-1\right) t^{2}\partial _{t}+\left( 
\frac{u^{2-m}}{m-1}-2\lambda tu\right) \partial _{u}\text{,\ }M_{4}=\left[
t\partial _{x}+\left( \frac{u^{2-m}}{m-1}-\lambda tu\right) \partial _{u}%
\right] e^{-\lambda \left( m-1\right) x}\text{,} \\
M_{5} &=&\left( \frac{-1}{\lambda \left( m-1\right) }\partial _{t}+\frac{%
u^{m}}{m-1}\partial _{u}\right) e^{\lambda \left( m-1\right) x},\text{ }%
M_{6}=-t\partial _{t}+\frac{u}{m-1}\partial _{u}, \\
M_{7} &=&\left( \frac{-1}{\lambda \left( m-1\right) }\partial _{x}+\frac{u}{%
m-1}\partial _{u}\right) e^{-\lambda \left( m-1\right) x}\text{,} \\
M_{8} &=&\left( \partial _{x}+\lambda \left( m-1\right) t\partial
_{t}-\lambda ^{2}\left( m-1\right) tu^{m}\partial _{u}\right) e^{\lambda
\left( m-1\right) x}\text{.}
\end{eqnarray*}%
\textbf{Case 3}. \ For $k=1-m\left( m\neq 1\right) $, equation $\left( \ref%
{34}\right) $ admits eight parameters group of projective transformations
spanned by the vector fields $M_{1}$,$M_{2}$,%
\begin{eqnarray*}
M_{3} &=&\left( -x^{2}+\frac{\lambda ^{2}\left( 1-m\right) ^{2}}{4}%
t^{4}\right) \partial _{x}-\left( xt+\frac{\lambda \left( 1-m\right) }{2}%
t^{3}\right) \partial _{t} \\
&&+\frac{1}{1-m}\left[ tu^{2-m}+\left( \frac{3}{2}\lambda \left( 1-m\right)
t^{2}-x\right) u+\lambda ^{2}\left( 1-m\right) ^{2}t^{3}u^{m}\right]
\partial _{u}, \\
M_{4} &=&\left( 3\lambda \left( 1-m\right) xt-\frac{\lambda ^{2}\left(
1-m\right) ^{2}}{2}t^{3}\right) \partial _{x}+\left( \frac{3}{2}\lambda
\left( 1-m\right) t^{2}-x\right) \partial _{t} \\
&&+\frac{1}{1-m}\left[ u^{2-m}+\left( 3\lambda \left( 1-m\right) x-\frac{3}{2%
}\lambda ^{2}\left( 1-m\right) ^{2}t^{2}\right) u^{m}\right] \partial _{u},
\\
M_{5} &=&\left( \frac{\lambda \left( 1-m\right) }{2}t^{3}-xt\right) \partial
_{x}-t^{2}\partial _{t}+\frac{1}{1-m}\left[ tu+\left( \frac{3}{2}\lambda
\left( 1-m\right) t^{2}-x\right) u^{m}\right] \partial _{u}, \\
M_{6} &=&t\partial _{x}+\frac{u^{m}}{1-m}\partial _{u},M_{7}=\left( x+\frac{%
\lambda \left( 1-m\right) }{2}t^{2}\right) \partial _{x}+\frac{1}{1-m}\left(
u+\lambda \left( 1-m\right) tu^{m}\right) \partial _{u}, \\
M_{8} &=&\left( x-\frac{\lambda \left( 1-m\right) }{2}t^{2}\right) \partial
_{x}+t\partial _{t}-\lambda tu^{m}\partial _{u}\text{.}
\end{eqnarray*}%
\textbf{Case 4}. \ For $k=\frac{m-1}{2}\left( m\neq 1\right) $, equation $%
\left( \ref{34}\right) $ admits eight parameters group of projective
transformations spanned by the vector field $M_{1}$,$M_{2}$,%
\begin{eqnarray*}
M_{3} &=&\left( xt-\frac{\lambda \left( m-1\right) }{4}x^{3}\right) \partial
_{x}+\left( t^{2}+\frac{\lambda ^{2}\left( m-1\right) ^{2}}{16}x^{4}\right)
\partial _{t} \\
&&+\frac{2}{m-1}\left[ \frac{\lambda ^{2}\left( m-1\right) ^{2}}{4}x^{3}u^{%
\frac{m+1}{2}}-\left( t+\frac{3}{4}\lambda \left( m-1\right) x^{2}\right)
u+tu^{\frac{3-m}{2}}\right] \partial _{u}, \\
M_{4} &=&\left( t-\frac{3}{4}\lambda \left( m-1\right) x^{2}\right) \partial
_{x}-\frac{\lambda ^{2}\left( m-1\right) ^{2}}{4}x^{3}\partial _{t} \\
&&+\frac{2}{m-1}\left[ \frac{3}{4}\lambda ^{2}\left( m-1\right) ^{2}x^{2}u^{%
\frac{m+1}{2}}-\frac{3}{2}\lambda \left( m-1\right) xu+u^{\frac{3-m}{2}}%
\right] \partial _{u}, \\
M_{5} &=&-x^{2}\partial _{x}-\left( xt-\frac{\lambda }{4}\left( m-1\right)
x^{3}\right) \partial _{t}+\frac{2}{m-1}\left[ \left( t+\frac{3}{4}\lambda
\left( m-1\right) x^{2}\right) u^{\frac{m+1}{2}}-xu\right] \partial _{u}, \\
M_{6} &=&x\partial _{x}+\left( t+\frac{\lambda }{4}\left( m-1\right)
x^{2}\right) \partial _{t}-\lambda xu^{\frac{m+1}{2}}\partial _{u}, \\
M_{7} &=&x\partial _{x}+\frac{\lambda }{2}\left( m-1\right) x^{2}\partial
_{t}+\frac{2}{m-1}\left( u-\lambda \left( m-1\right) xu^{\frac{m+1}{2}%
}\right) \partial _{u},\text{ }M_{8}=-x\partial _{x}+\frac{2}{m-1}u^{\frac{%
m+1}{2}}\partial _{u}\text{.}
\end{eqnarray*}%
\textbf{Case 5}. \ For $k=\frac{1-m}{2}\left( m\neq 1\right) $, equation $%
\left( \ref{34}\right) $ admits three parameters group of projective
transformations spanned by the closed vector field%
\begin{equation*}
M_{1},\text{\ }M_{2},\text{\ }M_{3}=\frac{3}{2}x\partial _{x}+t\partial _{t}+%
\frac{1}{1-m}u\partial _{u}\text{.}
\end{equation*}%
\textbf{Case 6}. \ For $k=\frac{m-1}{3}\left( m\neq 1\right) $, equation $%
\left( \ref{34}\right) $ admits three parameters group of projective
transformations spanned by the closed vector field%
\begin{equation*}
M_{1},\text{ }M_{2},\text{\ }M_{3}=-2x\partial _{x}-3t\partial _{t}+\frac{3}{%
m-1}u\partial _{u}\text{.}
\end{equation*}%
\textbf{Case 7}. \ For $m=1\left( k\neq 0\right) $, equation $\left( \ref{34}%
\right) $ admits eight parameters group of projective transformations
spanned by the closed vector field $M_{1}$,$M_{2}$,%
\begin{eqnarray*}
M_{3} &=&\lambda kx^{2}\partial _{x}-x\partial _{t}+\frac{1}{k}\left(
u^{k+1}+2\lambda kxu\right) \partial _{u},\text{ }M_{4}=\left[ -x\partial
_{t}+\frac{1}{k}\left( u^{k+1}+\lambda kxu\right) \partial _{u}\right]
e^{\lambda kt}, \\
M_{5} &=&\frac{1}{k}\left( \frac{-1}{\lambda }\partial _{x}+u^{1-k}\partial
_{u}\right) e^{-\lambda kt},\text{ }M_{6}=x\partial _{x}+\frac{u}{k}\partial
_{u}, \\
M_{7} &=&\frac{1}{k}\left( \frac{-1}{\lambda }\partial _{t}+u\partial
_{u}\right) e^{\lambda kt},\text{ }M_{8}=\lambda k\left( -x\partial _{x}+%
\frac{1}{\lambda k}\partial _{t}+\lambda xu^{1-k}\partial _{u}\right)
e^{-\lambda kt}\text{.}
\end{eqnarray*}
\end{theorem}

\begin{proof}
Using $g=u^{k}$, $h=u^{m}$ and equation $\left( \ref{38}\right) $ in $\left( %
\ref{36}\right) $ and simplifying we get%
\begin{gather}
\frac{1}{k}A_{x}u^{2k+1}+\lambda \left[ \frac{1}{k}\left( m-k-1\right)
A+\tau _{x}\right] u^{k+m}+\lambda \left[ \frac{1}{k}\left( m-1\right)
B+\tau _{t}\right] u^{m}  \notag \\
+\frac{\lambda }{k}\left( m+k-1\right) Cu^{m-k}+\frac{1}{k}\left(
A_{t}+B_{x}\right) u^{k+1}+\frac{1}{k}\left( B_{t}+C_{x}\right) u+\frac{1}{k}%
C_{t}u^{1-k}=0\text{.}  \label{42}
\end{gather}%
\textbf{Case 1}. \ Since all the coefficient in $\left( \ref{42}\right) $
are independent of $u$, we have $\frac{1}{k}\left( m-k-1\right) A+\tau
_{x}=0 $ which implies $A=0$ using equation $\left( \ref{39}\right) $, thus $%
\tau _{x}=0$ and $B_{x}=0$. \ Also $C=0$ implies $\xi _{t}=0$ and $B_{t}=0$
and so $B=\alpha _{3}$, a constant. \ Therefore, the required infinitesimal
transformations are%
\begin{equation}
\xi =\frac{k-m+1}{k}\alpha _{3}x+\alpha _{1},\text{ }\tau =\frac{1-m}{k}%
\alpha _{3}t+\alpha _{2},\text{ }\phi =\frac{1}{k}\alpha _{3}u  \label{43}
\end{equation}%
which proves the result.\newline
\textbf{Case 2}. \ Under this case, equation $\left( \ref{42}\right) $
becomes%
\begin{multline*}
\frac{1}{m-1}\left( A_{x}+\lambda \left( m-1\right) \tau _{x}\right)
u^{2m-1}+\frac{1}{m-1}\left[ \lambda \left( m-1\right) \left( B+\tau
_{t}\right) +A_{t}+B_{x}\right] u^{m} \\
+\frac{1}{m-1}\left( 2\lambda \left( m-1\right) C+B_{t}+C_{x}\right) u+\frac{%
1}{m-1}C_{t}u^{2-m}=0\text{,}
\end{multline*}%
which holds if and only if 
\begin{gather}
C\left( x,t\right) =C\left( x\right) \text{,}  \notag \\
A_{x}-\lambda \left( m-1\right) A=0\text{,}  \label{44} \\
\lambda \left( m-1\right) \xi _{x}+A_{t}+B_{x}=0\text{,}  \label{45} \\
2\lambda \left( m-1\right) C+B_{t}+C_{x}=0\text{.}  \label{46}
\end{gather}%
From $\left( \ref{41}\right) $, $\left( \ref{45}\right) $ and $\left( \ref%
{46}\right) $, we have%
\begin{gather*}
B=-\left( C^{\prime }\left( x\right) +2\lambda \left( m-1\right) C\right)
t+a\left( x\right) \text{,} \\
A=\frac{1}{2}\left[ C^{\prime \prime }\left( x\right) +\lambda \left(
m-1\right) C^{\prime }\left( x\right) \right] t^{2}-\left[ \lambda \left(
m-1\right) b^{\prime }\left( x\right) +a^{\prime }\left( x\right) \right]
t+d\left( x\right) \text{,}
\end{gather*}%
where $a\left( x\right) $, $b\left( x\right) $ and $d\left( x\right) $ are
abitrary functions. \ Consistency criterion for $A$ in equations $\left( \ref%
{44}\right) $ results into%
\begin{gather}
C^{\prime \prime \prime }\left( x\right) -\lambda ^{2}\left( m-1\right)
^{2}C^{\prime }\left( x\right) =0\text{,}  \notag \\
\lambda ^{2}\left( m-1\right) ^{2}b^{\prime }\left( x\right) -\lambda \left(
m-1\right) b^{\prime \prime }\left( x\right) +\lambda \left( m-1\right)
a^{\prime }\left( x\right) -a^{\prime \prime }\left( x\right) =0\text{,}
\label{47} \\
d^{\prime }\left( x\right) -\lambda \left( m-1\right) d\left( x\right) =0%
\text{.}  \notag
\end{gather}%
Using $\left( \ref{47}\right) $ with $\left( \ref{39}\right) $ and $\left( %
\ref{40}\right) $, and solving the resulting equations that follow, we get
the infinitesimal transformations to be%
\begin{eqnarray*}
\xi &=&\left( \alpha _{1}+\alpha _{2}e^{-\lambda \left( m-1\right) x}\right)
t-\frac{\alpha _{5}}{\lambda \left( m-1\right) }e^{-\lambda \left(
m-1\right) x}+\alpha _{7}e^{\lambda \left( m-1\right) x}+\alpha _{6}\text{,}
\\
\tau &=&\lambda \left( m-1\right) \alpha _{1}t^{2}-\frac{\alpha _{3}}{%
\lambda \left( m-1\right) }e^{\lambda \left( m-1\right) x}+\left( -\alpha
_{4}+\lambda \left( m-1\right) \alpha _{7}e^{\lambda \left( m-1\right)
x}\right) t+\alpha _{8}\text{,} \\
\phi &=&\frac{1}{m-1}\left[ 
\begin{array}{c}
\left( \alpha _{3}-\lambda ^{2}\left( m-1\right) ^{2}\alpha _{7}t\right)
u^{m}e^{\lambda \left( m-1\right) x} \\ 
+\left( -2\lambda \left( m-1\right) \alpha _{1}t+\alpha _{4}+\left( -\lambda
\left( m-1\right) \alpha _{2}t+\alpha _{5}\right) e^{-\lambda \left(
m-1\right) x}\right) u \\ 
+\left( \alpha _{1}+\alpha _{2}e^{-\lambda \left( m-1\right) x}\right)
u^{2-m}%
\end{array}%
\right] \text{.}
\end{eqnarray*}%
The Lie algebras are obtained from these infinitesimals. \ The closure of
these algebras is shown in the following commutator table where $a=\lambda
\left( m-1\right) $ and $V=aM_{6}-2M_{1}$.\bigskip \newline
$%
\begin{tabular}{|c|c|c|c|c|c|c|c|c|}
\hline
& $M_{1}$ & $M_{2}$ & $M_{3}$ & $M_{4}$ & $M_{5}$ & $M_{6}$ & $M_{7}$ & $%
M_{8}$ \\ \hline
$M_{1}$ & $0$ & $0$ & $0$ & $-aM_{4}$ & $aM_{5}$ & $0$ & $-aM_{7}$ & $aM_{8}$
\\ \hline
$M_{2}$ & $0$ & $0$ & $M_{1}-2aM_{6}$ & $-aM_{7}$ & $0$ & $-M_{2}$ & $0$ & $%
-a^{2}M_{5}$ \\ \hline
$M_{3}$ & $0$ & $-M_{1}+2aM_{6}$ & $0$ & $0$ & $\frac{M_{8}}{a}$ & $M_{3}$ & 
$M_{4}$ & $0$ \\ \hline
$M_{4}$ & $aM_{4}$ & $aM_{7}$ & $0$ & $0$ & $\frac{M_{1}}{a}+M_{6}$ & $M_{4}$
& $0$ & $aM_{3}$ \\ \hline
$M_{5}$ & $-aM_{5}$ & $0$ & $-\frac{M_{8}}{a}$ & $-\frac{M_{1}}{a}-M_{6}$ & $%
0$ & $-M_{5}$ & $-\frac{M_{2}}{a}$ & $0$ \\ \hline
$M_{6}$ & $0$ & $M_{2}$ & $-M_{3}$ & $-M_{4}$ & $M_{5}$ & $0$ & $0$ & $0$ \\ 
\hline
$M_{7}$ & $aM_{7}$ & $0$ & $-M_{4}$ & $0$ & $\frac{M_{2}}{a}$ & $0$ & $0$ & $%
V$ \\ \hline
$M_{8}$ & $-aM_{8}$ & $a^{2}M_{5}$ & $0$ & $-aM_{3}$ & $0$ & $0$ & $-V$ & $0$
\\ \hline
\end{tabular}%
\bigskip \medskip $\newline
\textbf{Case 7}. \ We have from $\left( \ref{42}\right) $,%
\begin{gather}
A=A\left( t\right) \text{,}  \notag \\
A^{\prime }\left( t\right) -2\lambda kA\left( t\right) +B_{x}=0\text{,}
\label{7.1} \\
\lambda k\tau _{t}+B_{t}+C_{x}=0\text{,}  \label{7.2} \\
\lambda kC+C_{t}=0\text{.}  \label{7.3}
\end{gather}%
Equations $\left( \ref{39}\right) $, $\left( \ref{7.1}\right) $ and $\left( %
\ref{7.2}\right) $ respectively imply%
\begin{eqnarray*}
\tau &=&-A\left( t\right) x+p\left( t\right) \text{,} \\
B &=&\left( 2\lambda kA\left( t\right) -A^{\prime }\left( t\right) \right)
x+q\left( t\right) \text{,} \\
C &=&\frac{1}{2}\left( A^{\prime \prime }\left( t\right) -\lambda kA^{\prime
}\left( t\right) \right) x^{2}-\left( \lambda kp^{\prime }\left( t\right)
+q^{\prime }\left( t\right) \right) x+r\left( t\right) \text{,}
\end{eqnarray*}%
where $p\left( t\right) $, $q\left( t\right) $ and $r\left( t\right) $ are
abitrary functions. \ Hence equation $\left( \ref{7.3}\right) $ gives%
\begin{gather}
A\left( t\right) =\alpha _{1}+\alpha _{2}e^{\lambda kt}+\alpha e^{-\lambda
kt}\text{,}  \notag \\
\lambda k\left( p^{\prime \prime }\left( t\right) +\lambda kp^{\prime
}\left( t\right) \right) +q^{\prime \prime }\left( t\right) +\lambda
kq^{\prime }\left( t\right) =0\text{,}  \label{7.4} \\
r\left( t\right) =\alpha _{3}e^{-\lambda kt}\text{.}  \notag
\end{gather}%
Equations $\left( \ref{40}\right) $, $\left( \ref{41}\right) $, $\left( \ref%
{7.4}\right) $ and their stability criteria lead to%
\begin{eqnarray*}
\alpha &=&0 \\
p\left( t\right) &=&\alpha _{6}+\alpha _{7}e^{-\lambda kt}-\frac{\alpha _{5}%
}{\lambda k}e^{\lambda kt} \\
q\left( t\right) &=&\alpha _{4}+\alpha _{5}e^{\lambda kt} \\
B\left( x,t\right) &=&\left( 2\lambda k\alpha _{1}+\lambda k\alpha
_{2}e^{\lambda kt}\right) x+\alpha _{4}+\alpha _{5}e^{\lambda kt} \\
C\left( x,t\right) &=&\alpha _{3}e^{-\lambda kt}+\left( \lambda k\right)
^{2}\alpha _{7}e^{-\lambda kt}x
\end{eqnarray*}%
Thus, the coefficients of the symmetry generator are%
\begin{eqnarray*}
\xi &=&\lambda k\alpha _{1}x^{2}+\left( \alpha _{4}-\lambda k\alpha
_{7}e^{-\lambda kt}\right) x-\frac{\alpha _{3}}{\lambda k}e^{-\lambda
kt}+\alpha _{8}\text{,} \\
\tau &=&-\left( \alpha _{1}+\alpha _{2}e^{\lambda kt}\right) x-\frac{\alpha
_{5}}{\lambda k}e^{\lambda kt}+\alpha _{6}+\alpha _{7}e^{-\lambda kt}\text{,}
\\
\phi &=&\frac{1}{k}\left[ 
\begin{array}{c}
\left( \alpha _{1}+\alpha _{2}e^{\lambda kt}\right) u^{k+1}+\left( \lambda
k\left( 2\alpha _{1}+\alpha _{2}e^{\lambda kt}\right) x+\alpha _{4}+\alpha
_{5}e^{\lambda kt}\right) u \\ 
+\left( \alpha _{3}+\left( \lambda k\right) ^{2}\alpha _{7}x\right)
e^{-\lambda kt}u^{1-k}%
\end{array}%
\right] \text{.}
\end{eqnarray*}%
The Lie algebras are obtained from the infinitesimals above and their
closure is shown in the following commutator table, $N=M_{2}-2\lambda kM_{6}$%
.\bigskip \newline
$%
\begin{tabular}{|c|c|c|c|c|c|c|c|c|}
\hline
& $M_{1}$ & $M_{2}$ & $M_{3}$ & $M_{4}$ & $M_{5}$ & $M_{6}$ & $M_{7}$ & $%
M_{8}$ \\ \hline
$M_{1}$ & $0$ & $0$ & $-N$ & $\lambda kM_{7}$ & $0$ & $M_{1}$ & $0$ & $%
\lambda ^{2}k^{2}M_{5}$ \\ \hline
$M_{2}$ & $0$ & $0$ & $0$ & $\lambda kM_{4}$ & $-\lambda kM_{5}$ & $0$ & $%
\lambda kM_{7}$ & $-\lambda kM_{8}$ \\ \hline
$M_{3}$ & $N$ & $0$ & $0$ & $0$ & $-M_{8}$ & $-M_{3}$ & $-M_{4}$ & $0$ \\ 
\hline
$M_{4}$ & $-\lambda kM_{7}$ & $-\lambda kM_{4}$ & $0$ & $0$ & $-\frac{M_{2}}{%
\lambda k}-M_{6}$ & $-M_{4}$ & $0$ & $-\lambda kM_{3}$ \\ \hline
$M_{5}$ & $0$ & $\lambda kM_{5}$ & $M_{8}$ & $\frac{M_{2}}{\lambda k}+M_{6}$
& $0$ & $M_{5}$ & $\frac{M_{1}}{\lambda k}$ & $0$ \\ \hline
$M_{6}$ & $-M_{1}$ & $0$ & $M_{3}$ & $M_{4}$ & $-M_{5}$ & $0$ & $0$ & $0$ \\ 
\hline
$M_{7}$ & $0$ & $-\lambda kM_{7}$ & $M_{4}$ & $0$ & $-\frac{M_{1}}{\lambda k}
$ & $0$ & $0$ & $2M_{2}-\lambda kM_{6}$ \\ \hline
$M_{8}$ & $-\lambda ^{2}k^{2}M_{5}$ & $\lambda kM_{8}$ & $0$ & $\lambda
kM_{3}$ & $0$ & $0$ & $-2M_{2}+\lambda kM_{6}$ & $0$ \\ \hline
\end{tabular}%
\bigskip $\newline
Cases $5$ and $6$ can be obtained by substituting respective values of $k$
in $\left( \ref{43}\right) $ while the remaining cases can be proved in a
similar way as did in case two.
\end{proof}

\section{Group Transformations}

Each of the infinitesimal generators $M_{i}$ in theorem \ref{damping} above
can be used to obtain corresponding one-parameter $\left( \varepsilon
\right) $ Lie group of point transformations $G_{i}$ through exponentiation.
\ To achieve this, we solved the following system of differential equations%
\begin{eqnarray*}
\frac{d\widetilde{x_{k}}\left( \varepsilon \right) }{d\varepsilon } &=&\xi
_{k}\left( \widetilde{x_{k}}\left( \varepsilon \right) ,\widetilde{u}\left(
\varepsilon \right) \right) \text{,} \\
\frac{d\widetilde{u}\left( \varepsilon \right) }{d\varepsilon } &=&\phi
\left( \widetilde{x_{k}}\left( \varepsilon \right) ,\widetilde{u}\left(
\varepsilon \right) \right) \text{,}
\end{eqnarray*}%
subject to the initial conditions%
\begin{equation*}
\left. \left( \widetilde{x_{k}},\widetilde{u}\right) \right\vert
_{\varepsilon =0}=\left( x_{k},u\right)
\end{equation*}%
for all $X_{i}:i=1,2,...,8$ and obtained the following (in addition to $%
G_{1}:\left( x,t,u\right) \longmapsto \left( x+\varepsilon ,t,u\right) $ and 
$G_{2}:\left( x,t,u\right) \longmapsto \left( x,t+\varepsilon ,u\right) $):%
\newline
\textbf{Case 1}\textit{.}\medskip \newline
$G_{3}:\left( x,t,u\right) \longmapsto \left( e^{\frac{k-m+1}{k}\varepsilon
}x,e^{\frac{1-m}{k}\varepsilon }t,e^{\frac{\varepsilon }{k}}u\right) $%
.\medskip \newline
\textbf{Case 2}\textit{\ }$\left( a=\lambda \left( m-1\right) \right) $%
.\medskip \newline
$G_{3}:\left( x,t,u\right) \longmapsto \left( x-\ln \sqrt[a]{1-a\varepsilon t%
},\frac{t}{1-a\varepsilon t},\sqrt[m-1]{\varepsilon \left( 1-a\varepsilon
t\right) +\left( 1-a\varepsilon t\right) ^{2}u^{m-1}}\right) $,\medskip $%
\newline
G_{4}:\left( x,t,u\right) \longmapsto \left( \ln \sqrt[a]{a\varepsilon
t+e^{ax}},t,\sqrt[m-1]{\frac{\varepsilon e^{-ax}+u^{m-1}}{1+a\varepsilon
te^{-ax}}}\right) $,\medskip $\newline
G_{5}:\left( x,t,u\right) \longmapsto \left( x,t-\frac{\varepsilon e^{ax}}{a}%
,\sqrt[1-m]{\varepsilon e^{ax}+u^{1-m}}\right) $,\medskip $\newline
G_{6}:\left( x,t,u\right) \longmapsto \left( x,e^{-\varepsilon }t,e^{\frac{%
\varepsilon }{m-1}}u\right) $,\medskip $\newline
G_{7}:\left( x,t,u\right) \longmapsto \left( \ln \sqrt[a]{-\varepsilon
+e^{ax}},t,e^{\lambda x}\sqrt[1-m]{-\varepsilon +e^{ax}}u\right) $,\medskip $%
\newline
G_{8}:\left( x,t,u\right) \longmapsto \left( -\ln \sqrt[a]{-a\varepsilon
+e^{-ax}},\frac{t}{1-a\varepsilon e^{ax}},\sqrt[1-m]{\frac{a^{2}\varepsilon t%
}{a\varepsilon -e^{-ax}}+u^{1-m}}\right) $.\medskip $\newline
$\textbf{Case 3}\textit{\ }$\left( a=\lambda \left( 1-m\right) \right) $%
.\medskip $\newline
G_{6}:\left( x,t,u\right) \longmapsto \left( x+\varepsilon t,t,\sqrt[1-m]{%
\varepsilon +u^{1-m}}\right) $,\medskip $\newline
G_{7}:\left( x,t,u\right) \longmapsto \left( \left( x+\frac{a}{2}%
t^{2}\right) e^{\varepsilon }-\frac{a}{2}t^{2},t,\sqrt[1-m]{e^{\varepsilon
}\left( u^{1-m}+2at\right) -2at}\right) $,\medskip $\newline
G_{8}:\left( x,t,u\right) \longmapsto \left( \left[ x+\frac{a}{2}t^{2}\left(
1-e^{\varepsilon }\right) \right] e^{\varepsilon },e^{\varepsilon }t,\sqrt[%
1-m]{a\left( 1-e^{\varepsilon }\right) t+u^{1-m}}\right) $.\medskip $\newline
$\textbf{Case 4}\textit{\ }$\left( a=\frac{\lambda \left( m-1\right) }{2}%
\right) $.\medskip $\newline
G_{6}:\left( x,t,u\right) \longmapsto \left( e^{\varepsilon }x,\left[ t+%
\frac{a}{2}\left( e^{\varepsilon }-1\right) x^{2}\right] e^{\varepsilon },%
\left[ a\left( e^{\varepsilon }-1\right) x+u^{\frac{1-m}{2}}\right] ^{\frac{2%
}{1-m}}\right) $,\medskip $\newline
G_{7}:\left( x,t,u\right) \longmapsto \left( e^{\varepsilon }x,t+\frac{a}{2}%
\left( e^{2\varepsilon }-1\right) x^{2},\left[ \left( a\left(
e^{2\varepsilon }-1\right) x+u^{\frac{1-m}{2}}\right) e^{-\varepsilon }%
\right] ^{\frac{2}{1-m}}\right) $,\medskip $\newline
G_{8}:\left( x,t,u\right) \longmapsto \left( e^{-\varepsilon }x,t,\left(
-\varepsilon +u^{\frac{1-m}{2}}\right) ^{\frac{2}{1-m}}\right) $.\medskip $%
\newline
$\textbf{Case 5}\textit{.}\medskip $\newline
G_{3}:\left( x,t,u\right) \longmapsto \left( e^{\frac{3}{2}\varepsilon
}x,e^{\varepsilon }t,e^{\frac{3}{1-m}\varepsilon }u\right) $.\medskip $%
\newline
$\textbf{Case 6}\textit{.}\medskip $\newline
G_{3}:\left( x,t,u\right) \longmapsto \left( e^{-2\varepsilon
}x,e^{-3\varepsilon }t,e^{\frac{3}{m-1}\varepsilon }u\right) $.\medskip $%
\newline
$\textbf{Case 7}\textit{.}\medskip $\newline
G_{3}:\left( x,t,u\right) \longmapsto \left( \frac{x}{1-\lambda \varepsilon
kx},t-\frac{\ln \left( 1-\lambda \varepsilon kx\right) }{\lambda k},\frac{u}{%
\sqrt[k]{1-\lambda \varepsilon kx-\varepsilon xu^{k}}}\right) $,\medskip $%
\newline
G_{4}:\left( x,t,u\right) \longmapsto \left( x,-\frac{\ln \left( \lambda
\varepsilon kx+e^{-\lambda kt}\right) }{\lambda k},\sqrt[k]{\frac{1+\lambda
\varepsilon kxe^{\lambda kt}}{1-\varepsilon e^{\lambda kt}u^{k}}}u\right) $%
,\medskip $\newline
G_{5}:\left( x,t,u\right) \longmapsto \left( x-\frac{\varepsilon }{\lambda k}%
e^{-\lambda kt},t,\sqrt[k]{\varepsilon e^{-\lambda kt}+u^{k}}\right) $%
,\medskip $\newline
G_{6}:\left( x,t,u\right) \longmapsto \left( xe^{\varepsilon },t,e^{\frac{%
\varepsilon }{k}}u\right) $,\medskip $\newline
G_{7}:\left( x,t,u\right) \longmapsto \left( x,-\frac{\ln \left(
-\varepsilon +e^{-\lambda kt}\right) }{\lambda k},\sqrt[k]{1-\varepsilon
e^{\lambda kt}}u\right) $,\medskip $\newline
G_{8}:\left( x,t,u\right) \longmapsto \left( \frac{x}{1+\lambda \varepsilon
ke^{-\lambda kt}},\frac{\ln \left( \lambda \varepsilon k+e^{\lambda
kt}\right) }{\lambda k},\sqrt[k]{\frac{\lambda ^{2}\varepsilon
k^{2}xe^{-\lambda kt}}{1+\lambda \varepsilon ke^{-\lambda kt}}+u^{k}}\right) 
$.\medskip $\newline
$From all the above calculated one parameter Lie groups of transformations,
the following theorems on transformed solutions hold true.

\begin{theorem}
If $u=\varphi \left( x,t\right) $ is an invariant solution of $\left( \ref%
{34}\right) $ obtained through the generators $M_{1}$ and $M_{2}$, then so
are the following functions:\newline
$G_{1}.\varphi \left( x,t\right) =\varphi \left( x-\varepsilon ,t\right) $,%
\newline
$G_{2}.\varphi \left( x,t\right) =\varphi \left( x,t-\varepsilon \right) $.
\end{theorem}

\begin{theorem}
If $u=\varphi \left( x,t\right) $ is an invariant solution of $\left( \ref%
{34}\right) $ obtained through the remaining generator of case one, then so
is the following function:\medskip \newline
$G_{3}.\varphi \left( x,t\right) =e^{-\frac{\varepsilon }{k}}\varphi \left(
e^{\frac{k-m+1}{k}\varepsilon }x,e^{\frac{1-m}{k}\varepsilon }t\right) $.
\end{theorem}

\begin{theorem}
If $u=\varphi \left( x,t\right) $ is an invariant solution of $\left( \ref%
{34}\right) $ obtained through the generators of case two, then so are the
following functions:\medskip \newline
$G_{3}.\varphi \left( x,t\right) =\sqrt[m-1]{\frac{\left[ \varphi \left(
x-\ln \sqrt[a]{1-a\varepsilon t},\frac{t}{1-a\varepsilon t}\right) \right]
^{m-1}-\varepsilon \left( 1-a\varepsilon t\right) }{\left( 1-a\varepsilon
t\right) ^{2}}}$,\medskip \newline
$G_{4}.\varphi \left( x,t\right) =\sqrt[m-1]{\left( 1+a\varepsilon
te^{-ax}\right) \left[ \varphi \left( \ln \sqrt[a]{a\varepsilon t+e^{ax}}%
,t\right) \right] ^{m-1}-\varepsilon e^{-ax}}$,\medskip \newline
$G_{5}.\varphi \left( x,t\right) =\sqrt[1-m]{\left[ \varphi \left( x,t-\frac{%
\varepsilon e^{ax}}{a}\right) \right] ^{1-m}-\varepsilon e^{ax}}$,\medskip 
\newline
$G_{6}.\varphi \left( x,t\right) =e^{-\frac{\varepsilon }{m-1}}\varphi
\left( x,e^{-\varepsilon }t\right) $,\medskip \newline
$G_{7}.\varphi \left( x,t\right) =e^{-\lambda x}\sqrt[m-1]{-\varepsilon
+e^{ax}}\varphi \left( \ln \sqrt[a]{-\varepsilon +e^{ax}},t\right) $%
,\medskip \newline
$G_{8}.\varphi \left( x,t\right) =\sqrt[1-m]{\left[ \varphi \left( -\ln \sqrt%
[a]{-a\varepsilon +e^{-ax}},\frac{t}{1-a\varepsilon e^{ax}}\right) \right]
^{1-m}-\frac{a^{2}\varepsilon t}{a\varepsilon -e^{-ax}}}$.
\end{theorem}

\begin{theorem}
If $u=\varphi \left( x,t\right) $ is an invariant solution of $\left( \ref%
{34}\right) $ obtained through the generators of case three, then so are the
following functions:\medskip \newline
$G_{6}.\varphi \left( x,t\right) =\sqrt[1-m]{\left[ \varphi \left(
\varepsilon t+x,t\right) \right] ^{1-m}-\varepsilon }$,\medskip \newline
$G_{7}.\varphi \left( x,t\right) =\sqrt[1-m]{\left\{ \left[ \varphi \left(
\left( x+\frac{a}{2}t^{2}\right) e^{\varepsilon }-\frac{a}{2}t^{2},t\right) %
\right] ^{1-m}+2at\right\} e^{-\varepsilon }-2at}$,\medskip \newline
$G_{8}.\varphi \left( x,t\right) =\sqrt[1-m]{\left[ \varphi \left( \left( x+%
\frac{a}{2}t^{2}\left( 1-e^{\varepsilon }\right) \right) e^{\varepsilon
},te^{\varepsilon }\right) \right] ^{1-m}-at\left( 1-e^{\varepsilon }\right) 
}$.
\end{theorem}

\begin{theorem}
If $u=\varphi \left( x,t\right) $ is an invariant solution $\left( \ref{34}%
\right) $ obtained through the generators of case four, then so are the
following functions:\medskip \newline
$G_{6}.\varphi \left( x,t\right) =\left[ \left[ \varphi \left(
xe^{\varepsilon },\left( t+\frac{a}{2}x^{2}\left( e^{\varepsilon }-1\right)
\right) e^{\varepsilon }\right) \right] ^{\frac{1-m}{2}}-ax\left(
e^{\varepsilon }-1\right) \right] ^{\frac{2}{1-m}}$,\medskip \newline
$G_{7}.\varphi \left( x,t\right) =\left[ e^{\varepsilon }\left[ \varphi
\left( xe^{\varepsilon },t+\frac{a}{2}x^{2}\left( e^{2\varepsilon }-1\right)
\right) \right] ^{\frac{1-m}{2}}-ax\left( e^{2\varepsilon }-1\right) \right]
^{\frac{2}{1-m}}$,\medskip \newline
$G_{8}.\varphi \left( x,t\right) =\left[ \left[ \varphi \left(
xe^{-\varepsilon },t\right) \right] ^{\frac{1-m}{2}}+\varepsilon \right] ^{%
\frac{2}{1-m}}$.
\end{theorem}

\begin{theorem}
If $u=\varphi \left( x,t\right) $ is an invariant solution of $\left( \ref%
{34}\right) $ obtained through the respective generator of case five and
six, then so are the following respective function:\medskip \newline
$G_{3}.\varphi \left( x,t\right) =e^{-\frac{\varepsilon }{1-m}}\varphi
\left( e^{\frac{3}{2}\varepsilon }x,e^{\varepsilon }t\right) $.$\medskip 
\newline
G_{3}.\varphi \left( x,t\right) =e^{-\frac{3}{m-1}\varepsilon }\varphi
\left( e^{-2\varepsilon }x,e^{-3\varepsilon }t\right) $.
\end{theorem}

\begin{theorem}
\label{th7}If $u=\varphi \left( x,t\right) $ is an invariant solution of $%
\left( \ref{34}\right) $ obtained through the generators of case seven, then
so are the following functions:\medskip \newline
$G_{3}.\varphi \left( x,t\right) =\sqrt[k]{\frac{\left( 1-\lambda
\varepsilon kx\right) }{1+\lambda \varepsilon x\left[ \varphi \left( \frac{x%
}{1-\lambda \varepsilon kx},t-\ln \sqrt[\lambda k]{1-\lambda \varepsilon kx}%
\right) \right] ^{k}}}\left[ \varphi \left( \frac{x}{1-\lambda \varepsilon kx%
},t-\ln \sqrt[\lambda k]{1-\lambda \varepsilon kx}\right) \right] $,$%
\medskip \newline
G_{4}.\varphi \left( x,t\right) =\frac{\varphi \left( x,-\ln \sqrt[\lambda k]%
{\lambda \varepsilon kx+e^{-\lambda kt}}\right) }{\sqrt[k]{1+\varepsilon
e^{\lambda kt}\left( \lambda kx+\left[ \varphi \left( x,-\ln \sqrt[\lambda k]%
{\lambda \varepsilon kx+e^{-\lambda kt}}\right) \right] ^{k}\right) }}$,$%
\medskip \newline
G_{5}.\varphi \left( x,t\right) =\sqrt[k]{\left[ \varphi \left( x-\frac{%
\varepsilon }{\lambda k}e^{-\lambda kt},t\right) \right] ^{k}-\varepsilon
e^{-\lambda kt}}$,$\medskip \newline
G_{6}.\varphi \left( x,t\right) =e^{-\frac{\varepsilon }{k}}\varphi \left(
e^{\varepsilon }x,t\right) $,$\medskip \newline
G_{7}.\varphi \left( x,t\right) =\frac{1}{\sqrt[k]{1-\varepsilon e^{\lambda
kt}}}\varphi \left( x,-\ln \sqrt[\lambda k]{-\varepsilon +e^{-\lambda kt}}%
\right) $,$\medskip \newline
G_{8}.\varphi \left( x,t\right) =\sqrt[k]{\left[ \varphi \left( \frac{x}{%
1+\lambda \varepsilon ke^{-\lambda kt}},\ln \sqrt[\lambda k]{\lambda
\varepsilon k+e^{\lambda kt}}\right) \right] ^{k}-\frac{\lambda
^{2}k^{2}\varepsilon xe^{-\lambda kt}}{1+\lambda \varepsilon ke^{-\lambda kt}%
}}$.
\end{theorem}

\begin{example}
This example highlights the importance of the above theorems. \ We start
with the trivial solution of equation $\left( \ref{34}\right) $, that is $%
u\left( x,t\right) =0$. \ By applying the transformations of theorem \ref%
{th7} on this solution we found that $G_{3}$, $G_{4}$, $G_{6}$, and $G_{7}$
all give no new solution but $G_{5}$ and $G_{8}$ respectively give $u=\sqrt[k%
]{-\varepsilon e^{-\lambda kt}}$ and $u=\sqrt[k]{-\frac{\lambda
^{2}k^{2}\varepsilon xe^{-\lambda kt}}{1+\lambda \varepsilon ke^{-\lambda kt}%
}}$ as the new non-trivial one parameter invariant solutions of $\left( \ref%
{34}\right) $. \ We next pick one of these solutions and apply the rest of
the transformations on it so as to generate more new solutions. \ Suppose we
pick $u\left( x,t\right) =\sqrt[k]{-\varepsilon e^{-\lambda kt}}$ which is
invariant under $G_{5}$, the actions of the rest of the $G_{i}$ are as
follows:\newline
$G_{7}:G_{5}\longrightarrow \sqrt[k]{-\varepsilon e^{-\lambda kt}\text{,}}%
\medskip $\newline
$G_{8}:G_{7}\longrightarrow \sqrt[k]{\frac{-\varepsilon -\lambda
^{2}k^{2}\varepsilon _{1}x}{\lambda k\varepsilon _{1}+e^{\lambda kt}}}$,$%
\medskip $\newline
$G_{4}:G_{8}\longrightarrow \sqrt[k]{\frac{\varepsilon +\lambda
^{2}k^{2}\varepsilon _{1}x}{\left( \varepsilon \varepsilon _{2}-1\right)
e^{\lambda kt}-\lambda k\varepsilon _{1}}}$,$\medskip $\newline
$G_{3}:G_{4}\longrightarrow \sqrt[k]{\frac{\left( 1-\lambda k\varepsilon
_{3}x\right) \left( \lambda ^{2}k^{2}\varepsilon _{1}x-\lambda k\varepsilon
\varepsilon _{3}x+\varepsilon \right) }{e^{\lambda kt}+\lambda \left(
1-\lambda k\varepsilon _{3}x\right) \left( \varepsilon \varepsilon
_{3}x-k\varepsilon _{1}\right) +\lambda ^{3}k^{2}\varepsilon _{1}\varepsilon
_{3}x^{2}}}$,$\medskip $\newline
$G_{6}:G_{3}\longrightarrow e^{-\frac{\varepsilon _{4}}{k}}\sqrt[k]{\frac{%
\left( 1-\lambda k\varepsilon _{3}e^{\varepsilon _{4}}x\right) \left(
\lambda ^{2}k^{2}\varepsilon _{1}e^{\varepsilon _{4}}x-\lambda k\varepsilon
\varepsilon _{3}e^{\varepsilon _{4}}x+\varepsilon \right) }{e^{\lambda
kt}+\lambda \left( 1-\lambda k\varepsilon _{3}e^{\varepsilon _{4}}x\right)
\left( \varepsilon \varepsilon _{3}e^{\varepsilon _{4}}x-k\varepsilon
_{1}\right) +\lambda ^{3}k^{2}\varepsilon _{1}\varepsilon
_{3}e^{2\varepsilon _{4}}x^{2}}}$.$\medskip $\newline
Hence we see that by applying the transformations of theorem \ref{th7} to $%
u\left( x,t\right) =0$, we obtained a non-trivial five-parameters solution%
\begin{equation*}
u\left( x,t\right) =e^{-\frac{\varepsilon _{4}}{k}}\sqrt[k]{\frac{\left(
1-\lambda k\varepsilon _{3}e^{\varepsilon _{4}}x\right) \left( \lambda
^{2}k^{2}\varepsilon _{1}e^{\varepsilon _{4}}x-\lambda k\varepsilon
\varepsilon _{3}e^{\varepsilon _{4}}x+\varepsilon \right) }{e^{\lambda
kt}+\lambda \left( 1-\lambda k\varepsilon _{3}e^{\varepsilon _{4}}x\right)
\left( \varepsilon \varepsilon _{3}e^{\varepsilon _{4}}x-k\varepsilon
_{1}\right) +\lambda ^{3}k^{2}\varepsilon _{1}\varepsilon
_{3}e^{2\varepsilon _{4}}x^{2}}\text{,}}
\end{equation*}%
that keeps equation $\left( \ref{34}\right) $ invariant. \ If we have chosen 
$u=\sqrt[k]{-\frac{\lambda ^{2}k^{2}\varepsilon xe^{-\lambda kt}}{1+\lambda
\varepsilon ke^{-\lambda kt}}}$ which is invariant under $G_{8}$, then the
actions of the rest of the $G_{i}$ will be as follows:\newline
$G_{7}:G_{8}\longrightarrow \sqrt[k]{\frac{\lambda ^{2}k^{2}\varepsilon x}{%
\left( 1-\lambda k\varepsilon \varepsilon _{1}\right) e^{\lambda kt}+\lambda
k\varepsilon }}$,\medskip \newline
$G_{5}:G_{7}\longrightarrow \sqrt[k]{\frac{\varepsilon _{2}\left( \lambda
k\varepsilon \varepsilon _{1}-1\right) -\lambda ^{2}k^{2}\varepsilon x}{%
\lambda k\varepsilon +\left( 1-\lambda k\varepsilon \varepsilon _{1}\right)
e^{\lambda kt}}}$,\medskip \newline
$G_{4}:G_{5}\longrightarrow \sqrt[k]{\frac{\varepsilon _{2}\left( \lambda
k\varepsilon \varepsilon _{1}-1\right) -\lambda ^{2}k^{2}\varepsilon x}{%
\lambda k\varepsilon +\left( 1-\lambda k\varepsilon \varepsilon _{1}\right)
\left( 1-\varepsilon _{2}\varepsilon _{3}\right) e^{\lambda kt}}}$,\medskip 
\newline
$G_{3}:G_{4}\longrightarrow \sqrt[k]{\frac{\left( 1-\lambda k\varepsilon
_{4}\right) \left[ \varepsilon _{2}\left( \lambda k\varepsilon \varepsilon
_{1}-1\right) \left( 1-\lambda k\varepsilon _{4}\right) -\lambda
^{2}k^{2}\varepsilon x\right] }{\lambda \left( 1-\lambda k\varepsilon
_{4}\right) \left[ k\varepsilon -\varepsilon _{2}\varepsilon _{4}\left(
1-\lambda k\varepsilon \varepsilon _{1}\right) x\right] -\lambda
^{3}k^{2}\varepsilon \varepsilon _{4}x^{2}+\left( 1-\lambda k\varepsilon
\varepsilon _{1}\right) \left( 1-\varepsilon _{2}\varepsilon _{3}\right)
e^{\lambda kt}}}$,\medskip \newline
$G_{6}:G_{3}\longrightarrow e^{-\frac{\varepsilon _{5}}{k}}\sqrt[k]{\frac{%
\left( 1-\lambda k\varepsilon _{4}\right) \left[ \varepsilon _{2}\left(
\lambda k\varepsilon \varepsilon _{1}-1\right) \left( 1-\lambda k\varepsilon
_{4}\right) -\lambda ^{2}k^{2}\varepsilon e^{\varepsilon _{5}}x\right] }{%
\lambda \left( 1-\lambda k\varepsilon _{4}\right) \left[ k\varepsilon
-\varepsilon _{2}\varepsilon _{4}\left( 1-\lambda k\varepsilon \varepsilon
_{1}\right) e^{\varepsilon _{5}}x\right] -\lambda ^{3}k^{2}\varepsilon
\varepsilon _{4}e^{2\varepsilon _{5}}x^{2}+\left( 1-\lambda k\varepsilon
\varepsilon _{1}\right) \left( 1-\varepsilon _{2}\varepsilon _{3}\right)
e^{\lambda kt}}\text{.}}$\medskip \newline
Therefore,%
\begin{equation*}
u\left( x,t\right) =\sqrt[k]{\frac{e^{-\varepsilon _{5}}\left( 1-\lambda
k\varepsilon _{4}\right) \left[ \varepsilon _{2}\left( \lambda k\varepsilon
\varepsilon _{1}-1\right) \left( 1-\lambda k\varepsilon _{4}\right) -\lambda
^{2}k^{2}\varepsilon e^{\varepsilon _{5}}x\right] }{\lambda \left( 1-\lambda
k\varepsilon _{4}\right) \left[ k\varepsilon -\varepsilon _{2}\varepsilon
_{4}\left( 1-\lambda k\varepsilon \varepsilon _{1}\right) e^{\varepsilon
_{5}}x\right] -\lambda ^{3}k^{2}\varepsilon \varepsilon _{4}e^{2\varepsilon
_{5}}x^{2}+\left( 1-\lambda k\varepsilon \varepsilon _{1}\right) \left(
1-\varepsilon _{2}\varepsilon _{3}\right) e^{\lambda kt}}}\text{,}
\end{equation*}%
is the new six-parameters invariant solution of equation $\left( \ref{34}%
\right) $.
\end{example}

\begin{itemize}
\item Transformations $G_{1}$ and $G_{2}$ were not used because they only
generate space and time translations respectively.

\item The order of applying the transformations is immaterial.

\item $G_{i}:G_{j}$ means the action of $G_{i}$ on the solution of $G_{j}$.
\end{itemize}

\begin{theorem}
If $g\left( u\right) =h\left( u\right) =u^{k}\left( k\neq 0\right) $\textbf{%
, }then all the above results hold as follows:\newline
1. \ For $k\neq \pm 1$, $k\neq \pm \frac{1}{2}$, and $k\neq $ $\frac{1}{3}$,
equation $\left( \ref{34}\right) $ admits three parameters group of
projective transformations spanned by the vector fields of case 1.$\newline
$2. \ For $k=1$, equation $\left( \ref{34}\right) $ admits eight parameters
group of projective transformations spanned by the vector fields of case 7.%
\newline
3. \ For $k=\frac{1}{2}$, equation $\left( \ref{34}\right) $ admits eight
parameters group of projective transformations spanned by the vector fields
of case 3.\newline
4. \ For $k=-1$, equation $\left( \ref{34}\right) $ admits eight parameters
group of projective transformations spanned by the vector fields of case 4.%
\newline
5. \ For $k=\frac{1}{3}$, equation $\left( \ref{34}\right) $ admits three
parameters group of projective transformations spanned by the vector fields
of case 5.\newline
6. \ For $k=-\frac{1}{2}$, equation $\left( \ref{34}\right) $ admits three
parameters group of projective transformations spanned by the vector field
of case 6.
\end{theorem}

\section{Reductions and Exact Invariant Solutions}

The reason behind going through all the tedious process of finding symmetry
generators of any given differential equation is mainly to use them for
obtaining symmetry reductions and possibly, exact solutions of the
underlying differential equation. \ In this section, we shall use the
symmetry generators of each of the classified cases in the previous section
to derive corresponding reduced equations and subsequently, obtain exact
invariant solutions of equation $\left( \ref{34}\right) $ whenever feasible.%
\newline
\textbf{Case 1}. \ The invariant surface condition of the generator $M_{3}=%
\frac{k-m+1}{k}x\partial _{x}+\frac{1-m}{k}t\partial _{t}+\frac{1}{k}%
u\partial _{u}$ is%
\begin{equation*}
\frac{k-m+1}{k}x\frac{\partial u}{\partial x}+\frac{1-m}{k}t\frac{\partial u%
}{\partial t}=\frac{1}{k}u
\end{equation*}%
with corresponding characteristics equations%
\begin{equation*}
\frac{k}{k-m+1}\frac{dx}{x}=\frac{k}{1-m}\frac{dt}{t}=k\frac{du}{u}\text{.}
\end{equation*}%
Solving these equations give two invariants $\psi =x^{\frac{1}{k-m+1}}t^{-%
\frac{1}{1-m}}$ and $\rho =ut^{-\frac{1}{1-m}}$. \ Since $\psi $ is
independent of $u$, then $\rho =F\left( \psi \right) $, $F$ arbitrary. \ So
the invariants become%
\begin{equation*}
\psi =x^{\frac{1}{k-m+1}}t^{-\frac{1}{1-m}},u=t^{\frac{1}{1-m}}F\left( \psi
\right) \text{.}
\end{equation*}%
By chain rule, we have 
\begin{eqnarray*}
u_{t} &=&\frac{1}{1-m}t^{\frac{m}{1-m}}\left( F-\psi F^{\prime }\right) 
\text{,} \\
u_{x} &=&\frac{1}{k-m+1}x^{\frac{m-k}{k-m+1}}F^{\prime }
\end{eqnarray*}%
which changes $\left( \ref{34}\right) $ to the reduced equation%
\begin{equation*}
F-\psi F^{\prime }+\frac{1-m}{k-m+1}\psi ^{m-k}F^{k}F^{\prime }+\lambda
\left( 1-m\right) F^{m}=0\text{.}
\end{equation*}%
This reduced equation is very difficult to solve for arbitrary $k$ and $m$,
but much can be obtained when appropriate choices of $k$ and $m$ are made. \
For example, if $m=k=2$, then the above equation becomes%
\begin{equation*}
F-\psi F^{\prime }-F^{2}F^{\prime }-\lambda F^{2}=0\text{.}
\end{equation*}%
After re-arrangement and simple quadrature, we obtain%
\begin{equation*}
F=\frac{-\left( \lambda \psi +A\right) \pm \sqrt{\left( \lambda \psi
+A\right) ^{2}+4\psi }}{2}
\end{equation*}%
where $A$ is a constant of integration. \ Hence the required invariant
solution of equation $\left( \ref{34}\right) $ with $m=k=2$ is%
\begin{equation*}
u\left( x,t\right) =\frac{-\left( \lambda xt+A\right) \pm \sqrt{\left(
\lambda xt+A\right) ^{2}+4xt}}{2t}\text{.}
\end{equation*}

\textbf{Case 2}. \ The characteristics equations of $M_{3}$ are%
\begin{equation*}
\frac{dx}{t}=\frac{dt}{\lambda \left( m-1\right) t^{2}}=\frac{\left(
m-1\right) du}{u^{2-m}-2\lambda \left( m-1\right) tu}
\end{equation*}%
having invariants $\psi =\frac{e^{\lambda \left( m-1\right) x}}{t}$ and $%
u=\left( \frac{1}{\lambda \left( m-1\right) t}+\frac{F\left( \psi \right) }{%
t^{2}}\right) ^{\frac{1}{m-1}}$. \ Substitution of $u$ and its respective
derivatives in $\left( \ref{34}\right) $ yield the reduced equation $\psi
F^{\prime }+F=0$ having solution $F=\frac{A}{\psi }$. \ Thus the required
invariant solution is%
\begin{equation*}
u=\left( \frac{1}{\lambda \left( m-1\right) t}+\frac{Ae^{-\lambda \left(
m-1\right) x}}{t}\right) ^{\frac{1}{m-1}}\text{.}
\end{equation*}%
The operators $M_{4}$ and $M_{6}$ also produced same result as $M_{3}$ with
the constant of integration $A$ replaced with $\frac{A}{\lambda \left(
1-m\right) }$ and $\frac{e^{A\lambda \left( m-1\right) }}{\lambda \left(
m-1\right) }$ respectively. \ Reduction with $M_{5}$ or $M_{8}$ yield the
invariant solution $\left( A+\lambda \left( m-1\right) t\right) ^{\frac{1}{%
1-m}}$ while $u=Ae^{-\lambda x}$ is obtained as a solution of $\left( \ref%
{34}\right) $ by using $M_{7}$.

\textbf{Case 3}. \ Two invariants of $M_{6}$ are $\psi =t$ and $u=\left( 
\frac{x}{t}+F\right) ^{\frac{1}{1-m}}$. \ These lead to the reduced equation 
$F^{\prime }+\frac{F}{\psi }+\lambda \left( 1-m\right) =0$ having solution $%
F=\left( \frac{A}{\psi }-\frac{\lambda \left( 1-m\right) }{2}\psi \right) $.
\ Thus the solution of $\left( \ref{34}\right) $ with $k=1-m$ is%
\begin{equation*}
u=\left( \frac{A+x}{t}-\frac{\lambda \left( 1-m\right) }{2}t\right) ^{\frac{1%
}{1-m}}\text{.}
\end{equation*}%
Solving the characteristic equations of $M_{7}$ give the invariants $\psi =t$
and $u=\left[ \left( 2x+at^{2}\right) F-at\right] ^{\frac{1}{1-m}}$ with the
second one obtained by solving the Bernoulli equation%
\begin{equation*}
\frac{du}{dx}-\frac{\lambda u}{a\left( 2x+at^{2}\right) }=\frac{\lambda
atu^{m}}{a\left( 2x+at^{2}\right) }
\end{equation*}%
where $a=\lambda \left( 1-m\right) $. \ The reduced equation obtained from
these invariants is $F^{\prime }+2F^{2}=0$ leading to the solution $F=\frac{1%
}{2\psi -A}$. \ Hence%
\begin{equation}
u=\left( \frac{2x+\lambda \left( 1-m\right) \left( A-t\right) t}{2t-A}%
\right) ^{\frac{1}{1-m}}\text{.}  \label{48}
\end{equation}%
Reduction with $M_{8}$ produced same result as that in $\left( \ref{48}%
\right) $ with $A=0$. \ $M_{8}$ also give additional solution $u=\left(
A-\lambda \left( 1-m\right) t\right) ^{\frac{1}{1-m}}$.\newline
We next use a linear combination of $M_{7}$ and $M_{8}$ to obtain a solution
of $\left( \ref{34}\right) $. \ The characteristic equations of $M_{7}+M_{8}$
are%
\begin{equation*}
\frac{dx}{2x}=\frac{dt}{t}=\left( 1-m\right) \frac{du}{u}
\end{equation*}%
with invariants $\psi =\frac{\sqrt{x}}{t}$, $u=\left( tF\right) ^{\frac{1}{%
1-m}}$. \ The reduced equation obtained is%
\begin{equation}
\left( F-2\psi ^{2}\right) F^{\prime }+2\psi \left( F-a\right) =0  \label{49}
\end{equation}%
where $a=\lambda \left( m-1\right) $. \ To solve this innocent looking
differential equation , we let 
\begin{equation}
F=\frac{a+2\psi ^{2}\left( z\left( \psi \right) -1\right) }{z\left( \psi
\right) }\text{.}  \label{50}
\end{equation}%
Thus 
\begin{eqnarray}
F^{\prime } &=&\frac{\left( 2\psi ^{2}-a\right) z^{\prime }+4\psi z\left(
z-1\right) }{z^{2}}\text{,}  \notag \\
a-F &=&\frac{\left( a-2\psi ^{2}\right) \left( z-1\right) }{z}\text{,}
\label{51} \\
F-2\psi ^{2} &=&\frac{a-2\psi ^{2}}{z}\text{.}  \notag
\end{eqnarray}%
Substituting $\left( \ref{50}\right) $ and $\left( \ref{51}\right) $ into $%
\left( \ref{49}\right) $ lead to the following seperable first order
differential equation%
\begin{equation*}
\frac{dz}{z\left( z-1\right) \left( z-2\right) }=\frac{2\psi d\psi }{a-2\psi
^{2}}
\end{equation*}%
having solution%
\begin{equation*}
\frac{z\left( z-2\right) }{\left( z-1\right) ^{2}\left( 2\psi ^{2}-a\right) }%
=A
\end{equation*}%
Without any lost of generality, we let $A=1$. \ Hence the solution of the
reduced equation $\left( \ref{49}\right) $ is%
\begin{equation*}
F^{2}-2\left( a+1\right) F+2\psi ^{2}+a\left( a+1\right) =0
\end{equation*}%
which is a quadratic equation in $F$ and so%
\begin{equation*}
F=\left( a+1\right) \pm \sqrt{a+1-2\psi ^{2}}\text{.}
\end{equation*}%
Therefore%
\begin{equation*}
u=\left[ \left[ \lambda \left( m-1\right) +1\right] t\pm \sqrt{\left[
\lambda \left( m-1\right) +1\right] t^{2}-2x}\right] ^{\frac{1}{1-m}}
\end{equation*}%
is the required invariant solution of $\left( \ref{34}\right) $ when $%
M_{7}+M_{8}$ is used as reducing generator.

The results for the remaning cases shall be presented in tabular form, $A$
is a constant wherever it appears.

\textbf{Case 4 }$\left( b=\lambda \left( m-1\right) \right) $.$\bigskip $

\begin{tabular}{|c|c|c|c|}
\hline
$\text{Subalgebra}$ & $\text{Invariants}$ & $\text{Reduced Equation}$ & $%
\text{Solution}$ \\ \hline
$M_{6}$ & $%
\begin{array}{c}
\psi =\frac{t}{x}-\frac{b}{4}x \\ 
u=\left( F+\frac{b}{2}x\right) ^{-\frac{2}{m-1}}%
\end{array}%
$ & $\left( F-\psi \right) F^{\prime }$ & $%
\begin{array}{c}
u=\left( \frac{t}{x}+\frac{b}{4}x\right) ^{-\frac{2}{m-1}} \\ 
u=\left( A+\frac{b}{4}x\right) ^{-\frac{2}{m-1}}%
\end{array}%
$ \\ \hline
$M_{7}$ & $%
\begin{array}{c}
\psi =\frac{b}{4}x^{2}-t \\ 
u=\left( \frac{F}{x}+\frac{b}{2}x\right) ^{-\frac{2}{m-1}}%
\end{array}%
$ & $F^{\prime }+1=0$ & $u=\left( \frac{A+t}{x}+\frac{b}{4}x\right) ^{-\frac{%
2}{m-1}}$ \\ \hline
$M_{6}-\frac{1}{2}M_{7}$ & $\psi =\frac{x^{2}}{t},u=\left( \frac{F}{x}%
\right) ^{\frac{2}{m-1}}$ & $\left( 2\psi F-\psi ^{2}\right) F^{\prime }+%
\frac{b}{2}F^{3}-F^{2}=0$ & - \\ \hline
\end{tabular}%
\bigskip

\textbf{Case 5 \& 6}.\bigskip

\begin{tabular}{|c|c|c|c|}
\hline
Case & $\text{Subalgebra}$ & $\text{Invariants}$ & $\text{Reduced Equation}$
\\ \hline
$5$ & $M_{3}$ & $\psi =\frac{\sqrt[3]{x}}{\sqrt{t}},u=\left( \sqrt{t}%
F\right) ^{\frac{2}{1-m}}$ & $\left( \frac{2F^{2}}{\psi ^{2}}-3\psi F\right)
F^{\prime }+3F^{2}+3\lambda \left( 1-m\right) =0$ \\ \hline
$6$ & $M_{3}$ & $\psi =\frac{\sqrt{x}}{\sqrt[3]{t}},u=\left( \frac{F}{\sqrt[3%
]{t}}\right) ^{\frac{3}{m-1}}$ & $\left( 3F-2\psi ^{2}\right) F^{\prime
}-2\psi F+2\lambda \left( m-1\right) \psi F^{4}=0$ \\ \hline
\end{tabular}%
\bigskip

\textbf{Case 7}.\bigskip

\begin{tabular}{|c|c|c|c|}
\hline
$\text{Subalgebra}$ & $\text{Invariants}$ & $\text{Reduced Equation}$ & $%
\text{Solution}$ \\ \hline
$M_{3}$ & $%
\begin{array}{c}
\psi =xe^{\lambda kt}, \\ 
u=\left( \frac{F}{x^{2}}-\frac{1}{\lambda kx}\right) ^{-\frac{1}{k}}%
\end{array}%
$ & $\psi F^{\prime }-F=0$ & $u=\left( \frac{Ae^{^{\lambda kt}}}{x}-\frac{1}{%
\lambda kx}\right) ^{-\frac{1}{k}}$ \\ \hline
$M_{4}$ & $\psi =x,u=\left( \frac{\lambda kxe^{-\lambda kt}F}{1-e^{-\lambda
kt}F}\right) ^{\frac{1}{k}}$ & $F^{\prime }=0$ & $u=\left( \frac{\lambda
Akxe^{-\lambda kt}}{1-Ae^{-\lambda kt}}\right) ^{\frac{1}{k}}$ \\ \hline
$M_{5}$ & $\psi =t,u=\left( F-\lambda kx\right) ^{\frac{1}{k}}$ & $F^{\prime
}=0$ & $u=\left( A-\lambda kx\right) ^{\frac{1}{k}}$ \\ \hline
$M_{6}$ & $\psi =t,u=\left( xF\right) ^{\frac{1}{k}}$ & $F^{\prime
}+F^{2}+\lambda kF=0$ & $u=\left( \frac{Bxe^{-\lambda kt}}{1-Be^{-\lambda kt}%
}\right) ^{\frac{1}{k}},B=\lambda ke^{\lambda kA}$ \\ \hline
$M_{7}$ & $\psi =x,u=\left( e^{-\lambda kt}F\right) ^{\frac{1}{k}}$ & $%
F^{\prime }=0$ & $u=\left( Ae^{-\lambda kt}\right) ^{\frac{1}{k}}$ \\ \hline
$M_{8}$ & $%
\begin{array}{c}
\psi =xe^{\lambda kt}, \\ 
u=\left( F-\lambda kx\right) ^{\frac{1}{k}}%
\end{array}%
$ & $F^{\prime }=0$ & $u=\left( A-\lambda kx\right) ^{\frac{1}{k}}$ \\ \hline
\end{tabular}%
\bigskip

Finally, we present results obtained through the generators $M_{1}=\partial
_{x}$ and $M_{2}=\partial _{t}$ which are known to have travelling waves
solution. \ The invariant of $c\partial _{x}+\partial _{t}$ is $x-ct$ for
any arbitrary constant $c\neq 0$. \ Since $\partial _{u}$ is missing, we can
only proceed by letting $u=F\left( \psi \right) $ where $\psi =x-ct$, hence $%
u_{t}=-cF^{\prime }$ and $u_{x}=F^{\prime }$. \ The obtained results are as
presented in the following table.\bigskip \newline
\begin{tabular}{|c|c|c|}
\hline
Case & $\text{Reduced Equation}$ & $\text{Solution}$ \\ \hline
1 & $\left( F^{k}-c\right) F^{\prime }+\lambda F^{m}=0$ & $u^{1-m}=\left( 
\frac{u^{k}}{k-m+1}-\frac{c}{1-m}\right) =A-\left( x-ct\right) $ \\ \hline
2 & $\left( F^{m-1}-c\right) F^{\prime }+\lambda F^{m}=0$ & $%
u^{1-m}e^{-cu^{1-m}}=Ae^{-\lambda \left( 1-m\right) \left( x-ct\right) }$ \\ 
\hline
3 & $\left( F^{m-1}-c\right) F^{\prime }+\lambda F^{m}=0$ & $u^{1-m}=c\pm 
\sqrt{c^{2}-2\lambda \left( 1-m\right) \left( x-ct\right) }$ \\ \hline
4 & $\left( F^{\frac{m-1}{2}}-c\right) F^{\prime }+\lambda F^{m}=0$ & $u^{%
\frac{m-1}{2}}=\frac{1\pm \sqrt{1-c\left( 1-m\right) \left[ A-\lambda \left(
x-ct\right) \right] }}{c}$ \\ \hline
5 & $\left( F^{\frac{1-m}{2}}-c\right) F^{\prime }+\lambda F^{m}=0$ & $2u^{%
\frac{3\left( 1-m\right) }{2}}-3cu^{1-m}=3\left( 1-m\right) \left[ A-\lambda
\left( x-ct\right) \right] $ \\ \hline
6 & $\left( F^{\frac{m-1}{3}}-c\right) F^{\prime }+\lambda F^{m}=0$ & $3u^{%
\frac{2\left( 1-m\right) }{3}}-2cu^{1-m}=2\left( 1-m\right) \left[ A-\lambda
\left( x-ct\right) \right] $ \\ \hline
7 & $\left( F^{k}-c\right) F^{\prime }+\lambda F=0$ & $%
u^{-ck}e^{u^{k}}=Ae^{-\lambda k\left( x-ct\right) }$ \\ \hline
\end{tabular}%
\bigskip

\begin{remark}
All the invariant solutions obtained in this section can be subjected to the
transformations of the previous section to generate new solutions.
\end{remark}

\section{Conclusion}

In this paper, we have used symmetry analysis to perform classifications and
subsequently, exhibit many invariant solutions for the damped inviscid
Burger's equation $\left( \ref{34}\right) $. \ Though various researchers
have used this equation to model important physical phenomena, I have search
the literature but did not see any work that deal with its analysis using
Lie method as presented in this paper, hence this is the first among the
series of work dedicated to the analysis of equation $\left( \ref{34}\right) 
$.

\end{document}